\documentclass[10pt]{amsart}
\usepackage[style=alphabetic]{biblatex} 
\usepackage{tikz}
\usetikzlibrary{angles}
\usepackage{amssymb}
\usepackage{amsthm}
\usepackage[all]{xy}
\usepackage{microtype}
\usepackage{xcolor}
\usepackage{hyperref}
\usepackage[nameinlink]{cleveref}
\usepackage{graphicx}

\makeatletter
    
    \@addtoreset{equation}{section}
  \makeatother

\hypersetup{
 colorlinks,
 linkcolor={teal},
 citecolor={teal},
 urlcolor={teal}
}

\calclayout

\DeclareFieldFormat
  [article,book,inbook,incollection,inproceedings,patent,thesis,unpublished]
  {title}{\emph{#1\isdot}}

\theoremstyle{plain}
	\newtheorem{theorem}{Theorem}
	\newtheorem{lemma}[theorem]{Lemma}

\theoremstyle{definition} 
	\newtheorem{remark}[theorem]{Remark}
	
	\newtheorem{example}[theorem]{Example}

\def\sgn{\mathop{\mathrm{sgn}}}

\begin{document}
\title{A new upper bound for mutually touching infinite cylinders}
\author[J. Koizumi]{Junnosuke Koizumi}
\address{RIKEN iTHEMS, Wako, Saitama 351-0198, Japan}
\email{junnosuke.koizumi@riken.jp}

\date{\today}
\thanks{}

\begin{abstract}
Let $N$ denote the maximum number of congruent infinite cylinders that can be arranged in $\mathbb{R}^3$ so that every pair of cylinders touches each other.
Littlewood posed the question of whether $N=7$, which remains unsolved.
In this paper, we prove that $N\leq 18$, improving the previously known upper bound of $24$ established by A.\! Bezdek.
\end{abstract}

\maketitle
\setcounter{tocdepth}{1}

\enlargethispage*{20pt}
\thispagestyle{empty}

\section{Introduction}

The following problem was originally posed by Littlewood \cite{Littlewood}:
What is the maximum number of congruent infinite cylinders that can be arranged in $\mathbb{R}^3$ so that any two of them are touching? Is it 7?

This problem can be interpreted as determining the maximum number $N$ of lines in $\mathbb{R}^3$ such that the distance between any two lines is exactly 1.
As Littlewood noted, heuristic arguments based on degrees of freedom suggest that $N = 7$.
A configuration of 6 cylinders can be constructed relatively easily; see e.g. \cite[2.4, Problem 14]{BMP}.
In the early 1990s, Kuperberg proposed a configuration of 8 cylinders, but this was later shown by Ambrus and A.\! Bezdek \cite{Ambrus_Bezdek} not to satisfy the required condition. The currently known best bounds are $7 \leq N \leq 24$. The upper bound was given by A.\! Bezdek \cite{Bezdek} through geometric arguments, while the lower bound was established by Boz\'oki, Lee, and R\'onyai \cite{BLR} using computer-assisted proof techniques.
In this paper, we give the following improved upper bound:

\begin{theorem}\label{main}
    Let $L_1,\dots,L_n$ be lines in $\mathbb{R}^3$ such that
    $$
    d(L_i,L_j):=\min_{x\in L_i,y\in L_j}\|x-y\|=1\quad(1\leq i<j\leq n).
    $$
    Then, we have $n\leq 18$.
\end{theorem}

Our strategy is as follows.
For any two directed lines $L,L'$ in $\mathbb{R}^3$ that are not coplanar, we define their \emph{chirality} $\varepsilon(L,L')\in \{\pm1\}$ as in the figure below.
We have $\varepsilon(L,L')=\varepsilon(L',L)$.
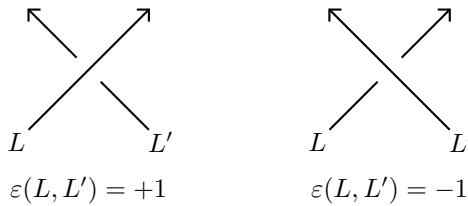
\begin{center}
\begin{tikzpicture}[thick, scale=0.8]
  \draw (-1,-1) -- (1,1);
  \draw(1,-1) -- (0.2, -0.2);
  \draw (-0.2, 0.2) -- (-1, 1);
  \draw (-1, 0.8) -- (-1,1) -- (-0.8, 1);
  \draw (1, 0.8) -- (1,1) -- (0.8, 1);
  \node at (-1.2,-1.2) {$L$};
  \node at (1.2,-1.2) {$L'$};
  \node at (0,-2) {$\varepsilon(L,L')=+1$};
  \begin{scope}[shift={(5,0)}]
  \draw (1,-1) -- (-1,1);
  \draw(-1,-1) -- (-0.2, -0.2);
  \draw (0.2, 0.2) -- (1, 1);
  \draw (-1, 0.8) -- (-1,1) -- (-0.8, 1);
  \draw (1, 0.8) -- (1,1) -- (0.8, 1);
  \node at (0,1.2) {};
  \node at (-1.2,-1.2) {$L$};
  \node at (1.2,-1.2) {$L'$};
  \node at (0,-2) {$\varepsilon(L,L')=-1$};
  \end{scope}
\end{tikzpicture}
\end{center}
Suppose that $L_1,\dots,L_n$ are directed lines in $\mathbb{R}^3$ such that no two of them are coplanar.
Then, we obtain a signed complete graph with $n$ vertices such that the sign on the edge $(i,j)$ is given by $\varepsilon(L_i,L_j)$.
We call this signed complete graph the \emph{chirality graph} of $\{L_1,\dots,L_n\}$.
We prove the following structural result on the chirality graph:

\begin{theorem}\label{directed}
    Let $L_1,\dots,L_n$ be directed lines in $\mathbb{R}^3$ such that no two of them are coplanar.
    If $d(L_i,L_j)=1$ holds for $1\leq i<j\leq n$, then the chirality graph of $\{L_1,\dots,L_n\}$ does not contain a monochromatic $K_5$.
\end{theorem}

It is easy to deduce \Cref{main} from \Cref{directed} using that the Ramsey number $R(4,4)$ is $18$.
We note that Theorem 2 is mentioned in a paper by the physicists P.\! Pikhitsa and S.\! Pikhitsa \cite{PP}, but their proof is not mathematically rigorous.
Also, it is worth noting that a similar problem involving cylinders of finite length is known as Gardner’s puzzle \cite[Problem 12.1]{Gardner}, and the maximum number of such cylinders that can mutually touch is also unknown.

\subsection*{Acknowledgement}
The author is grateful to Tomoaki Abuku for introducing this problem under the name ``The Middle-Aged Cheers Problem.''\footnote{The name comes from the following metaphor: middle-aged people, lacking youthful energy, prefer to clink glasses with everyone in a single toast, all at once. If each glass is modeled as a cylinder, the problem reduces to finding a configuration where all cylinders are mutually touching. This humorous name was coined by Takeaki Uno.}
The author is grateful to Yuhi Kamio for providing valuable comments on the draft of this paper. The author also thanks Yugo Takanashi for engaging in various discussions on this problem.

\section{An algebraic lemma}
For non-negative integers $a,b,c$, we define a polynomial function $f_{abc}\colon \mathbb{R}^3\to \mathbb{R}$ by
$$
    f_{abc}(x,y,z)=x^ay^bz^c.
$$
An easy computation shows that we have
    $$
        \langle v,w\rangle^n = \sum_{a+b+c=n}\dfrac{n!}{a!b!c!}f_{abc}(v)f_{abc}(w)
    $$
for any $v,w\in \mathbb{R}^3$, where $\langle {-},{-}\rangle$ denotes the standard inner product on $\mathbb{R}^3$.
We prepare the following algebraic lemma, which will be used in the proof of \Cref{directed}:

\begin{lemma}\label{algebraic}
    Let $n\geq 2$. Suppose that $v_1,\dots,v_n\in \mathbb{R}^3$ are unit vectors such that no two of them are parallel.
    Then, the symmetric matrix $S=(\|v_i\times v_j\|)_{i,j}$ is non-singular and has signature $(1,n-1)$.
\end{lemma}

\begin{proof}
    For the all-one vector $\mathbf{1}=(1,\dots,1)\in \mathbb{R}^n$, we have $\mathbf{1}^TS\mathbf{1}=\sum_{i,j}\|v_i\times v_j\|>0$.
    Let $H=\{(t_i)_i\in \mathbb{R}^n\mid \sum_i t_i=0\}$.
    It suffices to show that $S$ is negative definite on $H$.
    We have $\|v_i\times v_j\|=\sqrt{1-\langle v_i,v_j\rangle^2}$.
    The Taylor expansion of $\sqrt{1-x^2}$ is given by
    $$
        \sqrt{1-x^2} = 1-\sum_{k=1}^\infty \dfrac{1}{4^k(2k-1)}\binom{2k}{k} x^{2k},
    $$
    which converges for $-1\leq x\leq 1$.
    Therefore, the entries $s_{ij}$ of $S$ are given by
    \begin{align*}
        s_{ij} &{}= 1-\sum_{k=1}^\infty \dfrac{1}{4^k(2k-1)}\binom{2k}{k} \langle v_i,v_j\rangle^{2k}\\
        &{}=1-\sum_{k=1}^\infty\sum_{a+b+c=2k}C^{(k)}_{abc}f_{abc}(v_i)f_{abc}(v_j).\quad \left(C^{(k)}_{abc}:=\dfrac{1}{4^k(2k-1)}\binom{2k}{k}\dfrac{(2k)!}{a!b!c!}>0\right)
    \end{align*}
    For any vector $t=(t_i)_i\in \mathbb{R}^n$, we have
    \begin{align*}
        t^TSt = \sum_{i,j} s_{ij}t_it_j=\sum_{i,j}t_it_j - \sum_{i,j}\sum_{k=1}^\infty\sum_{a+b+c=2k}C^{(k)}_{abc}f_{abc}(v_i)f_{abc}(v_j)t_it_j.
    \end{align*}
    If $t\in H$, then the first term is $0$ and hence
    \begin{align*}
        t^TSt &{}= -\sum_{k=1}^\infty\sum_{a+b+c=2k}C^{(k)}_{abc}\left(\sum_{i=1}^nf_{abc}(v_i)t_i\right)^2\leq 0.
    \end{align*}
    This shows that $S$ is negative semi-definite on $H$.
    If $t^TSt=0$, then we have
    $$
        \sum_{i=1}^n f_{abc}(v_i)t_i=0
    $$
    for any $k\geq 0$ and any $a,b,c\geq 0$ with $a+b+c=2k$.
    This implies that for any polynomial $F\in \mathbb{R}[x,y,z]$ which has only even-degree terms, we have
    \begin{align}\label{eq1}
        \sum_{i=1}^n F(v_i)t_i=0.
    \end{align}
    For each $i\in \{1,2,\dots,n\}$, there is a linear function $f_i\colon \mathbb{R}^3\to \mathbb{R}$ such that $f_i(v_i)= 0$ and $f_i(v_j)\neq 0\ (i\neq j)$.
    Applying \eqref{eq1} for $F_i=\prod_{i\neq j}f_i^2$, we obtain $t_i=0$.
    This shows that $S$ is negative definite on $H$.
\end{proof}

\begin{remark}
    In the proof of \Cref{algebraic}, we expanded $\|v_i\times v_j\|$ as a power series in $\langle v_i,v_j\rangle$.
    Instead, we could also expand $\|v_i\times v_j\|$ by Legendre polynomials $P_\ell(\langle v_i,v_j\rangle)$ and use the addition theorem of spherical harmonic functions,
    $$
    P_\ell(\langle x,y\rangle)=\dfrac{4\pi}{2\ell+1}\sum_{-\ell\leq m\leq \ell} \overline{Y_\ell^m(x)}Y_\ell^m(y),
    $$
    to obtain the same result.
\end{remark}

\section{Proof of the main theorem}

Let $L,L'$ be two directed lines in $\mathbb{R}^3$ that are not coplanar.
Write $L=\mathbb{R}v+w$ and $L'=\mathbb{R}v'+w'$, where $v,v'\in \mathbb{R}^3$ are the unit vectors representing the direction of the line, and the vectors $w,w'\in \mathbb{R}^3$ are uniquely determined modulo $\mathbb{R}v$ (resp. $\mathbb{R}v'$).
We define the \emph{chirality} $\varepsilon(L,L')\in \{\pm 1\}$ by
$$
\varepsilon(L,L')=\sgn\langle v\times v',w-w'\rangle.
$$
One can easily see that this agrees with the intuitive definition of $\varepsilon(L,L')$ given in the introduction.
By definition, we have $\varepsilon(L,L')=\varepsilon(L',L)$.

Suppose that $L_1,\dots,L_n$ are directed lines in $\mathbb{R}^3$ such that no two of them are coplanar.
We define the \emph{chirality graph} of $\{L_1,\dots,L_n\}$ to be the signed complete graph on the vertex set $\{1,2,\dots,n\}$ whose sign on the edge $(i,j)$ is given by $\varepsilon(L_i,L_j)$.
If we reverse the direction of a line $L_i$, then the sign of the edges adjacent to the vertex $i$ will be reversed.

\begin{example}
    Consider the following configuration of directed lines:
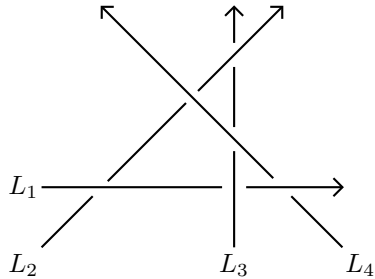
\begin{center}
\begin{tikzpicture}[thick, scale=0.8]
  \draw (0,1) -- (3,1);
  \draw (3.4,1) -- (5,1);

  \draw (0,0) -- (0.85,0.85);
  \draw (1.1,1.1) -- (2.4,2.4);
  \draw (2.6,2.6) -- (4,4);

  \draw (3.2,0) -- (3.2,1.6);
  \draw (3.2,2) -- (3.2,3);
  \draw (3.2,3.4) -- (3.2,4);

  \draw (1,4) -- (3.85, 1.15);
  \draw (4.15,0.85) -- (5,0);

  \draw (4.85,0.85) -- (5,1) -- (4.85, 1.15);
  \draw (4,3.8) -- (4,4) -- (3.8,4);
  \draw (3.05,3.85) -- (3.2, 4) -- (3.35, 3.85);
  \draw (1,3.8) -- (1,4) -- (1.2, 4);

  \node at (-0.3,1) {$L_1$};
  \node at (-0.3,-0.3) {$L_2$};
  \node at (3.2,-0.3) {$L_3$};
  \node at (5.3,-0.3) {$L_4$};

  \node at (0,4.2) {};
\end{tikzpicture}
\end{center}
    In this case, the chirality graph of $\{L_1,\dots,L_4\}$ is given as follows, where a solid line (resp. dashed line) represents an edge with sign $+1$ (resp. $-1$).
\begin{center}
\begin{tikzpicture}
  \draw[thick] (0,0) -- (2,0);
  \draw[dashed] (2,0) -- (2,2);
  \draw[thick] (2,2) -- (0,2);
  \draw[thick] (0,2) -- (0,0);
  \draw[dashed] (0,0) -- (2,2);
  \draw[dashed] (0,2) -- (2,0);
  \node at (0,2) [above left] {$1$};
  \node at (0,0) [below left] {$2$};
  \node at (2,0) [below right] {$3$};
  \node at (2,2) [above right] {$4$};
  \fill[black] (0,0) circle (2pt);
  \fill[black] (0,2) circle (2pt);
  \fill[black] (2,0) circle (2pt);
  \fill[black] (2,2) circle (2pt);
\end{tikzpicture}
\end{center}
\end{example}

\begin{example}
Boz\'oki, Lee, and R\'onyai \cite{BLR} constructed two distinct configurations of $7$ directed lines $L_1,\dots,L_7$ in $\mathbb{R}^3$ such that $d(L_i,L_j)=2$ holds for every $1\leq i<j\leq 7$.
The corresponding chirality graphs are given as in the figure below.
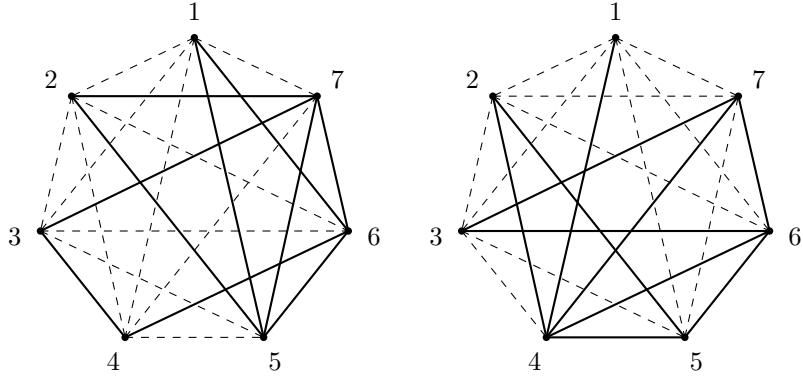
\begin{center}
\begin{tikzpicture}[scale=0.7]
  \foreach \i/\angle in {
    1/90,
    2/141,
    3/193,
    4/244,
    5/-64,
    6/-13,
    7/39
  }{
    \coordinate (\i) at (\angle:3cm);
    \node at (\angle:3.5cm) {\i};
    \fill[black] (\i) circle (2pt);
  }

  \foreach \i/\j in {1/5, 1/6, 2/5, 2/7, 3/4, 3/7, 4/6, 5/6, 5/7, 6/7}{
    \draw[thick] (\i) -- (\j);
  }
  \foreach \i/\j in {1/2, 1/3, 1/4, 1/7, 2/3, 2/4, 2/6, 3/5, 3/6, 4/5, 4/7}{
    \draw[dashed] (\i) -- (\j);
  }
  
  \begin{scope}[xshift=8cm]
  \foreach \i/\angle in {
    1/90,
    2/141,
    3/193,
    4/244,
    5/-64,
    6/-13,
    7/39
  }{
    \coordinate (\i) at (\angle:3cm);
    \node at (\angle:3.5cm) {\i};
    \fill[black] (\i) circle (2pt);
  }

  \foreach \i/\j in {1/4, 2/4, 2/5, 3/6, 3/7, 4/5, 4/6, 4/7, 5/6, 6/7}{
    \draw[thick] (\i) -- (\j);
  }
  \foreach \i/\j in {1/2, 1/3, 1/5, 1/6, 1/7, 2/3, 2/6, 2/7, 3/4, 3/5, 5/7}{
    \draw[dashed] (\i) -- (\j);
  }
  \end{scope}
\end{tikzpicture}
\end{center}
After reversing the direction of some of the lines, both graphs become isomorphic to the following graph:
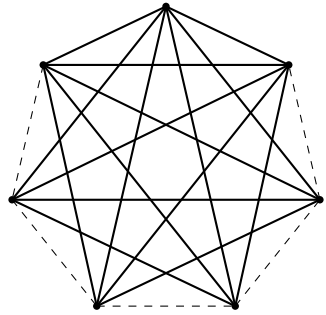
\begin{center}
\begin{tikzpicture}[scale=0.7]
  \foreach \i/\angle in {
    1/90,
    2/141,
    3/193,
    4/244,
    5/-64,
    6/-13,
    7/39
  }{
    \coordinate (\i) at (\angle:3cm);
    \fill[black] (\i) circle (2pt);
  }

  \foreach \i/\j in {1/2, 1/3, 1/4, 1/5, 1/6, 1/7, 2/4, 2/5, 2/6, 2/7, 3/5, 3/6, 3/7, 4/6, 4/7, 5/7}{
    \draw[thick] (\i) -- (\j);
  }
  \foreach \i/\j in {2/3, 3/4, 4/5, 5/6, 6/7}{
    \draw[dashed] (\i) -- (\j);
  }
  \node at (0,3.1) {};
\end{tikzpicture}
\end{center}
\end{example}

\begin{proof}[Proof of \Cref{directed}]
Assume, for the sake of contradiction, that the chirality graph of $\{L_1,\dots,L_n\}$ contains a monochromatic $K_5$.
Without loss of generality, we may assume that $n=5$ and that $\varepsilon(L_i,L_j)=+1$ for $1\leq i<j\leq 5$.
Write $L_i=\mathbb{R}v_i+w_i$, where $v_i\in \mathbb{R}^3$ is the unit vector representing the direction of $L_i$, and the vector $w_i\in \mathbb{R}^3$ is uniquely determined modulo $\mathbb{R}v_i$.
By our assumption, no two of the vectors $v_1,\dots,v_5$ are parallel.
For $1\leq i<j\leq 5$, we have
\begin{align*}
d(L_i,L_j)&{}=\min_{s,t\in \mathbb{R}}\|(sv_i-tv_j)+(w_i-w_j)\|\\
&{}=\biggl|\biggl\langle \dfrac{v_i\times v_j}{\|v_i\times v_j\|}, w_i-w_j\biggr\rangle\biggr|,
\end{align*}
so our assumption that $d(L_i,L_j)=1$ shows that
\begin{align*}
    \langle v_i\times v_j, w_i-w_j\rangle = \pm \|v_i\times v_j\|.
\end{align*}
The sign of the left hand side is $\varepsilon(L_i,L_j)$ by definition.
Therefore, our assumption that $\varepsilon(L_i,L_j)=+1$ shows that the right hand side is $\|v_i\times v_j\|$.
Let $S=(s_{ij})$ be a $5\times 5$ matrix whose entries are given by
$$
    s_{ij} = \langle v_i\times v_j, w_i-w_j\rangle = \|v_i\times v_j\|.
$$
By \Cref{algebraic}, the matrix $S$ is non-singular and has signature $(1,4)$.
On the other hand, if we set $q_i=w_i\times v_i\in \mathbb{R}^3$, then we have
\begin{align*}
    s_{ij} = \langle v_i\times v_j, w_i\rangle - \langle v_i\times v_j, w_j\rangle = \langle q_i, v_j\rangle + \langle v_i,q_j\rangle.
\end{align*}
Therefore, if we define a $6\times 6$ symmetric matrix $B$ by
$$
B=\begin{pmatrix}
    0&I_3\\I_3&0
\end{pmatrix},
$$
where $I_3$ is the $3\times 3$ identity matrix, then we have
$$
s_{ij} = (q_i^T\ v_i^T)B\begin{pmatrix}
    q_j\\v_j
\end{pmatrix}.
$$
Define a linear map $i\colon \mathbb{R}^5\to \mathbb{R}^6$ by $e_i\mapsto (q_i,v_i)$.
The above formula shows that $S$ is the matrix representing the bilinear form $i^*B$ on $\mathbb{R}^5$.
Since the signature of $B$ is $(3,3)$, this shows that the number of negative eigenvalues of $S$ cannot exceed $3$.
This contradicts the previous conclusion that $S$ has signature $(1,4)$.
\end{proof}

\begin{proof}[Proof of \Cref{main}]
Let $L_1,\dots,L_n$ be lines in $\mathbb{R}^3$ such that
$d(L_i,L_j)=1$ holds for $1\leq i<j\leq n$.
First, suppose that there is a pair of parallel lines among $L_1,\dots,L_n$.
Without loss of generality, we may assume that
\[
L_1=\mathbb{R}e_1+\dfrac{1}{2}e_2,\quad L_2=\mathbb{R}e_1-\dfrac{1}{2}e_2.
\]
If there exists an $i\geq 3$ that satisfies
\[
L_i = \mathbb{R}e_1\pm \dfrac{\sqrt{3}}{2}e_3,
\]
then there does not exist a line that is at distance $1$ from all of $L_1, L_2, L_i$ and hence $n \leq 3$.
Otherwise, the lines $L_3,\dots,L_n$ must be contained in $z=\pm 1$, and one can easily see that $n\leq 4$ in this case.

Next, suppose that no two of the lines $L_1,\dots,L_n$ are parallel.
We fix a direction of $L_i$ for each $i\in \{1,2,\dots,n\}$.
By reversing the direction of $L_2,\dots,L_n$ if necessary, we may assume that $\varepsilon(L_1,L_i)=+1$ for $i\in \{2,3,\dots,n\}$.
Consider the chirality graph of $\{L_2,\dots,L_n\}$.
By \Cref{directed}, this signed graph does not contain a monochromatic $K_4$; otherwise, possibly after reversing the direction of $L_1$, we obtain a monochromatic $K_5$ inside the chirality graph of $\{L_1,\dots,L_n\}$.
By Ramsey's theorem, this implies that $n-1 < R(4,4)=18$ and hence $n\leq 18$.
\end{proof}

\begin{remark}
In this paper, we proved that $N \leq 18$ using only the constraint that the chirality graph does not contain a monochromatic $K_5$, but it may be possible to improve this bound by identifying additional forbidden subgraphs.
For example, P.\! Pikhitsa and S.\! Pikhitsa \cite{PP} suggest, without providing a rigorous proof, that the following signed graph called $P_{250}$ is likely to be a forbidden subgraph.

\begin{center}
\begin{tikzpicture}[scale=0.7]
  \foreach \i/\angle in {
    1/90,
    2/141,
    3/193,
    4/244,
    5/-64,
    6/-13,
    7/39
  }{
    \coordinate (\i) at (\angle:3cm);
    \node at (\angle:3.5cm) {\i};
    \fill[black] (\i) circle (2pt);
  }

  \foreach \i/\j in {1/2, 1/3, 1/4, 1/5, 1/6, 1/7, 2/3, 2/4, 2/5, 2/6, 2/7, 3/5, 3/6, 4/6, 4/7, 5/7}{
    \draw[thick] (\i) -- (\j);
  }
  \foreach \i/\j in {3/4, 3/7, 4/5, 5/6, 6/7}{
    \draw[dashed] (\i) -- (\j);
  }
\end{tikzpicture}
\end{center}
\end{remark}

\printbibliography

\end{document}